\documentclass[11pt, a4paper]{article}
\usepackage{amsmath}
\usepackage[T1]{fontenc}
\usepackage[a4paper, total={6in, 9in}]{geometry}

\usepackage{graphicx}
    \graphicspath{{./images}}
\usepackage{tikz}
    \usetikzlibrary{calc, fadings, decorations.pathreplacing, calligraphy, positioning}

\usepackage{indentfirst,csquotes}
\usepackage{amssymb}
\usepackage{xcolor, hyperref}
    \hypersetup{ colorlinks=true, linkcolor=black, filecolor=black, urlcolor=black }
\usepackage{cleveref}

\usepackage{bm}
\usepackage{array}
\usepackage{multirow}

\usepackage{comment}

\newcommand{\A}{\mathcal{A}}
\newcommand{\Mop}{\mathcal{M}} 
\newcommand{\Lop}{\mathcal{L}} 
\newcommand{\inv}[1]{{#1}{\!\!\raisebox{.7ex}{\ensuremath{^{-1}}}}}
\newcommand{\R}{\mathbb{R}} 
\newcommand{\Z}{\mathbb{Z}} %
\newcommand{\W}{\mathcal{W}}

\newcommand{\f}{\bm{f}} 

\newcommand{\p}{\bm{p}}

\newcommand{\x}{\!\times\!}

\newcommand*{\Ones}{\text{\usefont{U}{bbold}{m}{n}1}}

\newcommand{\diag}{\mathrm{diag}}
\newcommand{\rank}{\mathrm{rank}}

\newcommand{\hb}{\rule[2pt]{10pt}{0.5pt}} 

\newlength{\imSz}
\newlength{\imShift}

\usepackage{amsthm}
\newtheorem{proposition}{Proposition}
\newtheorem*{remark}{Remark}

\newcommand{\email}[1]{\protect\href{mailto:#1}{#1}}

\begin{document}
\title{MultiResolution Low-Rank Regularization of Dynamic Imaging Problems}
\author{Tommi Heikkilä%
\thanks{%
LUT University, School of Engineering Science, Finland}\\[5pt]%
\href{https://orcid.org/0000-0001-5505-8136}{ORCID: 0000-0001-5505-8136}\\
\email{tommi.heikkila@lut.fi}%
}

\maketitle

\begin{abstract}
MultiResolution Low-Rank decomposition is formulated for regularization of dynamic image sequences. The decomposition applies a local low-rank decomposition on a sequence of discrete wavelet transforms. Its effective formulation as a regularization functional is discussed and numerically tested for dynamic X-ray tomography in comparison to other low-rank methods. The results suggest it is similar to traditional locally low-rank decomposition but produces less severe block artifacts.
\end{abstract} 

\bigskip

\section{Introduction}

In dynamic imaging problems such as magnetic resonance imaging (MRI) and computed tomography (CT) of moving targets it is often better to drastically shorten the measurement time since noisy and undersampled data is easier to correct than the underlying unknown movement. Over the years many \emph{low-rank methods} have been proposed and shown to be effective in resolving the static components of sequential data \cite{haldar2010spatiotemporal,zhao2010low,gao2011robust,tremoulheac2014dynamic}. If the imaging problem at hand is linear, the slowly changing and static parts of the image sequence can be reconstructed independently and more robustly. This also lowers the complexity of reconstructing the dynamic components and adds robustness since a rough outline of the object of interest is already provided, assuming the changes are not too drastic.

Local low-rank (LLR) methods consider only small patches of the images at once and try to find low-rank approximations for the smaller regions. Originally introduced for dynamic MRI \cite{trzasko2011local}, the method has also been used in magnetic resonance fingerprinting (MRF) \cite{lima2019sparsity}. Many of these low-rank methods also include sparsity promoting terms, either at the same time \cite{tremoulheac2014dynamic,lima2019sparsity} or in separate components \cite{gao2011robust,otazo2015low,ravishankar2017low}. See also \cite{kazantsev20154d} for different patch based temporal regularization method.

Recently many low-rank representations (LRRs) have also been developed in the context of tensors \cite{kilmer2011factorization,du2022enhanced,liu2025dynamic}, in order to extend the familiar tools and ideas from matrices (i.e. 2-dimensional arrays) to higher dimensional objects. This way vital underlying structure of the data is still preserved in dynamic (location and time), hyperspectral (location and frequency) and phase-space (location and momentum) applications, just to name a few. Furthermore in data processing and machine learning, extracting some lower dimensional subspace or manifold from hugely large and complicated data sets is often helpful in keeping the computations and storage manageable. In this work the perspective is limited only to the traditional low-rank matrix factorizations (namely the singular value decomposition (SVD)), but there is no inherent reason why the proposed method could not be extended to tensors as in \cite{liu2023image,liu2025dynamic}.

In this work a locally low-rank decomposition is formulated by applying the patched low-rank decomposition on the wavelet domain rather than the usual spatial domain. This idea has already been used in (hyperspectral) image denoising and deblurring \cite{palsson2014hyperspectral,zhao2021wavelet}, where it evolved from local wavelet coefficient thresholding methods. Alternate approach of wavelet-denoising the principal components has also been proposed even earlier \cite{chen2010denoising}.

Thanks to the multiscale nature of the wavelet decomposition, the low-rank decomposition will also cover multiple scales. Moreover, due to the smooth and overlapping wavelet elements, the patches should have less distinct boundary artifacts. We show that this multiresolution low-rank (MRLR) decomposition is unitary and briefly showcase its use as a regularization term with numerical examples in dynamic X-ray tomography application.

\section{Low-rank methods}
We consider a sequence of time-dependent inverse problems  of the form
\begin{equation} \label{eq:IPs_for_t}
    \A_t f_t = m_t + \varepsilon_t, \ \text{for} \ t = 1,2,\dots,T.
\end{equation}
Here each $f_t \in \R^N$ is a $N_r \times N_c$ image flattened into vector with $N = N_r N_c$ values in total. Extension to 3D volumes would follow in similar fashion. For simplicity, we denote the value at $i$'th row and $j$'th column of the image by $f_t[i, j]$, even though formally $f_t$ is a one dimensional vector.

We can stack each problem in \cref{eq:IPs_for_t} by denoting $\f = \big( f_t \big)_{t=1}^T$, $\bm{m} + \bm{\varepsilon} = \big( m_t + \varepsilon_t \big)_{t=1}^T$ and $\bm{\A} = \diag\big( \A_1, \dots , \A_T\big)$ the block diagonal operator. Then we can write \eqref{eq:IPs_for_t} as $\bm{\A} \f = \bm{m} + \bm{\varepsilon}$ andthe least squares term is then
\begin{align*}
    \big\| \bm{\A} \f - \bm{m} \big\|^2_2 = \sum_{t=1}^T \big\| \A_t f_t - m_t \big\|^2_2.
\end{align*}

Since each time step $t$ is currently independent, we regularize the problem by requiring that the evolution of $\bm{f}$ over time must have a low-rank structure.

We form the \emph{Casorati matrix} \cite{haldar2010spatiotemporal} using the associated operator $\Mop$ given as
\begin{align} \label{eq:casorati}
    \Mop : \bm{f} &\longmapsto F =\begin{bmatrix}
        f_1 & f_2 & \dots  & f_T
    \end{bmatrix} = \begin{bmatrix}
        f_1[1,1] & f_2[1,1] & \dots  & f_T[1,1] \\
        f_1[2,1] & f_2[2,1] & \dots  & f_T[2,1] \\
         \vdots  &  \vdots  & \ddots &  \vdots  \\
        f_1[N_r, N_c] & f_2[N_r, N_c] & \dots & f_T[N_r, N_c] \\
    \end{bmatrix}, 
\end{align}
where $F \in \R^{N \times T}$. 
Since this is simply a reordering of the terms, it can be reverted with
\begin{align}
     \inv{\Mop} : F & \longmapsto \f.   
\end{align}
We note also that for any $\f, \bm{g} \in \R^{NT}$ we have
\begin{align}
    \langle \f, \bm{g} \rangle &= \langle \Mop \f, \Mop \bm{g}  \rangle_F, \label{eq:frobenius_identity}
\end{align}
where $\langle \cdot, \cdot \rangle_F$ is the Frobenius inner product.

We could attempt penalizing the rank of the matrix $F = \Mop \f$, but this corresponds to minimizing the $\ell^0$-"norm" of the singular values. The usual remedy from compressed sensing and other sparsity promoting regularization methods is to use the convex $\ell^1$-norm instead. This leads to penalizing the \emph{nuclear norm} of the matrix $F$:
\begin{align}
    \big\| F \big\|_* = \sum_{i=1}^r \big| \sigma(F)_i \big| = \big\| \sigma(F) \big\|_1,
\end{align}
where $r = \rank(F) \leq \min\lbrace N, T \rbrace$ and $\sigma(F) = \diag(\Sigma) \in \R^r$ denotes the singular values of $F$ given by the singular value decomposition (SVD):
\begin{align} \label{eq:SVD}
    F = U \Sigma V^T = \begin{bmatrix}
        | & | &  & | \\
        u_1 & u_2 & \dots & u_r \\
        | & | &  & |
    \end{bmatrix} 
    \begin{bmatrix}
        \sigma(F)_1 & \\
        & \ddots & \\
        & & \sigma(F)_r
    \end{bmatrix}
    \begin{bmatrix}
        \hb & v_1^T & \hb \\
        \hb & v_2^T & \hb \\
         & \vdots & \\
        \hb & v_r^T & \hb
    \end{bmatrix}.
\end{align}

Hence the global low-rank (GLR) method is given by minimizing the functional
\begin{align} \label{eq:GLR}
    {GLR} \big( \f \big) = \big\| \bm{\A} \f - \bm{m} \big\|^2_2 + \lambda \left\| \Mop \f  \right\|_*,
\end{align}
where $\lambda > 0$ is the regularization parameter.

In order to minimize \cref{eq:GLR} effectively, we need to analyze the low-rank operator properly. Let $\diag$ denote the linear operator which takes the diagonal values from any matrix. For rectangular $m \times n$ matrices, if $m > n$, the output is zero-padded to be in $\R^m$. Formally we can write $\diag(M) = (M \odot I) \Ones$, where $\odot$ is the Hadamard product (elementwise product of matrix elements) and $\Ones \in \R^r$ a column vector of all ones. Finally for any $y \in \R^r$ we can also define $\diag^*(y) = (y \Ones^T) \odot I$, which gives a diagonal matrix with elements of $y$ on the diagonal.

Now consider
\begin{align} \label{eq:opL}
    \Lop : \f \longmapsto \diag \big( U^T \Mop ( \bm{f} ) V \big).
\end{align}
By definition $\Lop$ extracts the diagonal matrix $\Sigma$ from the SVD of $\Mop \f$ and we obtain a vector of $r$ singular values with the $\diag$ operator. It is easy to verify that $\Lop: \R^{NT} \rightarrow \R^r$ is a linear operator, but note that it only yields the singular values of the specific vector $\f$ since the unitary matrices $U$ and $V$ depend on the SVD of $\f$. However, we have the following result independent of the particular $\f \in \R^{NT}$.

\begin{proposition} \label{prop:L_is_unitary}
The operator $\Lop$ as defined in \eqref{eq:opL} is unitary and for any $y \in \R^r$
\begin{align} \label{eq:Ladj}
    \Lop^* y = \inv{\Mop}\left( U \diag^*(y) V^T \right).
\end{align}
\end{proposition}
\begin{proof}
The result follows easily by considering the Frobenius inner product, for which $\langle A B, C \rangle_F = \langle B, A^T C \rangle_F$ for any compatible sized matrices $A, B$ and $C$.

Let $\bm{x} = \big( x_t \big)_{t=1}^T \in \R^{NT}$ and $y\in \R^r$ be arbitrary and consider
\begin{align*}
    \langle \Lop \bm{x}, y \rangle &= \langle \diag \big( U^T \Mop (\bm{x}) V \big), y \rangle
    =  \langle  U^T \Mop (\bm{x}) V, \diag^* (y) \rangle_F \\
    &= \langle \Mop (\bm{x}) , U \diag^* (y) V^T \rangle_F 
    = \langle \bm{x}, \inv{\Mop} \big( U \diag^*(y) V^T \big) \rangle
\end{align*}
Therefore \cref{eq:Ladj} holds. Then we can check that
\begin{align*}
    \Lop \Lop^* y &=  \diag \Big( U^T \Mop \big( \inv{\Mop} ( U \diag^*(y) V^T ) \big) V \Big) \\
    &= \diag \big( U^T U \diag^*(y) V^T  V \big) = \big( (y \Ones^T) \odot I \odot I \big) \Ones = y,
\end{align*}
and
\begin{align*}
    \Lop^* \Lop \bm{x} &= \inv{\Mop} \big( U \diag^* \big( \diag( U^T \Mop ( \bm{x} ) V ) \big) V^T \big) \\
    &= \inv{\Mop} \big( U U^T \Mop ( \bm{x} ) V V^T \big) = \inv{\Mop} \big( \Mop ( \bm{x} ) \big) = \bm{x}.
\end{align*}
This concludes the proof.
\end{proof}

Proposition~\ref{prop:L_is_unitary} gives an alternative definition to \cref{eq:GLR} as
\begin{align} \label{eq:GLR_alt}
    \textit{GLR} \big( \f \big) = \big\| \bm{\A} \f - \bm{m} \big\|^2_2 + \lambda \left\| \Lop \f  \right\|_1,
\end{align}
which can be minimized with relatively standard methods.

\begin{remark}
An easier way of defining the $\diag$-operator would be to simply use a vector of ones $\Ones \in \R^r$. For diagonal matrices the operations $\diag(\Sigma)$ and $\Sigma \Ones$ would give identical results, but the adjoints would be different since $\diag^*(\diag(\Sigma)) = \Sigma$, but $\Sigma \Ones \Ones^T \neq \Sigma$ and $\Lop$ would not be unitary but $r$-times redundant.
\end{remark}

\subsection{Locally low-rank methods}
By construction the \emph{global low-rank} method in \cref{eq:GLR}  can only provide a very coarse approximation of the image sequence $\bm{f}$ since each principal component is of same size as $f_t$. And since the SVD is limited by the dimensions $N$ and $T$ (whichever is smaller), the number of components may be very limited, even when $\bm{f}$ has a very high spatial resolution or is a 3D volume, for example.

One way to overcome these limitations is to divide each $f_t$ into $K$ patches of size $p \times p$ (or more generally size $p_1 \times \dots \times p_d$ in $d$ dimensions) \cite{trzasko2011local}. We assume the patches are \emph{non-overlapping}. We can then consider a sequence of $K$ patched Casorati matrices, independently for each patch, as
\begin{align}
    \Mop_{p \times p} : \bm{f} &\longmapsto \left( F_k \right)_{k=1}^K, 
\end{align}
where for example, the first patch is 
\begin{align}
     F_1 &= \begin{bmatrix}
         f_1[1,1] & f_2[1,1] & \dots & f_T[1,1] \\
          \vdots  &  \vdots  & \ddots &  \vdots \\
         f_1[p,1] & f_2[p,1] & \dots & f_T[p,1] \\
         f_1[1,2] & f_2[1,2] & \dots & f_T[1,2] \\
          \vdots  &  \vdots  & \ddots &  \vdots \\
         f_1[p,p] & f_2[p,p] & \dots & f_T[p,p]
     \end{bmatrix}   
\end{align}
and so on. This process is also illustrated in \cref{fig:casoratis}.

\begin{figure}[bth]
    \centering
    \begin{tikzpicture}[scale=1]
    
    \begin{scope}[scale=0.5] 
    \foreach \x/\w/\s/\b in {4.5cm/50/.5cm/80, 4cm/30/0cm/65, 3.5cm/10/.5cm/50}{
        \filldraw[gray!\w!white, draw=black] (\x,\x) rectangle +(6cm,-6cm);
        \draw[step=1cm, xshift=\s, yshift=\s] ($(\x-\s,\x-\s)$) grid +(6cm,-6cm);
        \draw[very thick, red!\b!gray] (\x,\x) rectangle +(2cm,-2cm); 
        \draw[very thick, purple!\b!gray] ($(\x,\x-2cm)$) rectangle +(2cm,-2cm); 
        \draw[very thick, blue!\b!gray] ($(\x+4cm,\x-4cm)$) rectangle +(2cm,-2cm); 
    } 
    \draw[decorate, decoration = {calligraphic brace}, thick] (3.3,1.5) -- +(0,2) node[left, pos=0.5] {$p$};
    \draw[decorate, decoration = {calligraphic brace}, thick] (3.5,3.7) -- +(2,0) node[above=2pt, pos=0.5, fill=white, fill opacity=.7, text opacity=1, inner sep=1pt] {$p$};
    \node at (7cm,5.25cm) {$f_t, \ t = 1,\dots,T$};
    \end{scope}

    \draw[very thick, ->, >=stealth] (5.8,0.7) -- +(1.5,0) node[above, pos=.5] {$\Mop_{p \times p}$};

    \begin{scope}[xshift=7.5cm, yshift=0.5cm, scale=0.5] 
    \foreach \i/\j/\col/\s/\t/\k in {2.6cm/3.4cm/blue/-0.4cm/0.4cm/K, 1.8cm/2.2cm/purple/-0.2cm/0.2cm/2, 1cm/1cm/red/0cm/0cm/1}{ 
        \foreach \x/\w in {0cm/10, 1cm/30, 2cm/50}{ 
            \fill[gray!\w!white] ($(\x + \i, \j)$) rectangle +(1cm,-4cm);
        }
        \draw[step=1cm, xshift=\s, yshift=\t] ($(\i - \s, \j - \t)$) grid +(3cm,-4cm);
        \draw[very thick, \col!80!gray] ($(\i, \j)$) rectangle +(3cm,-4cm); 
    }
    \node at (3cm,4cm) {$F_k, \ k = 1, \dots, K$};
    \draw[decorate, decoration = {calligraphic brace}, thick] (0.8cm,-3.0cm) -- +(0,4cm) node[left, pos=0.5] {$p^2$};
    \draw[decorate, decoration = {calligraphic brace}, thick] (4cm,-3.2cm) -- +(-3cm,0) node[below, pos=0.5] {$T$};
    \end{scope}
    \end{tikzpicture}
    \caption{Illustration how the patched Casorati matrices are formed.}
    \label{fig:casoratis}
\end{figure}
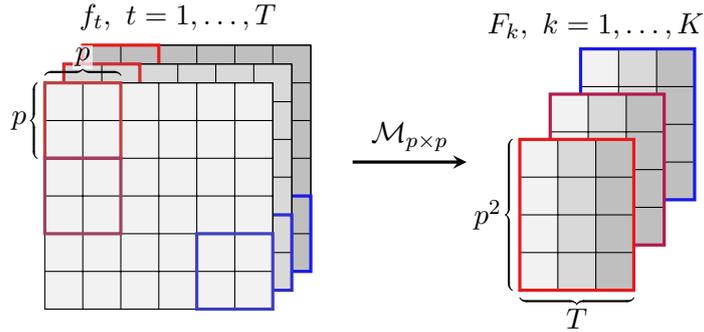

Hence the local low-rank (LLR) method is given by

\begin{align} \label{eq:LLR}
    \textit{LLR} \big( \bm{f} \big) = \big\| \bm{\A} \bm{f} - \bm{m} \big\|^2_2 + \lambda \sum_{k=1}^K \left\| \big( \Mop_{p \times p} \bm{f} \big)_k \right\|_*,
\end{align}
where $\lambda > 0$ is the regularization parameter. Defining a patched variant $\Lop_{p}$ of the unitary operator $\Lop$ would be straight forward (with either square or more general rectangular patches) but the construction is omitted here.

Changing the patch sizes gives a lot of control over the regions where each low-rank approximation is carried out. Notice also that the patches can have different sizes as long as each patch is present in every time step $f_t$ since the patches are separated in to the columns of matrices $F_k$ in the sequence, and $F_k$ and $F_l$ need not have the same size for $k \neq l$. Let $\p(k)$ denote the size of the $k$'th patch from now on.

However, this simple separation suffers from distinct artifacts at the boundaries of the patches since each is handled independently of the neighboring patches. This can be partly resolved with \emph{overlapping} patches, but this makes the associated operator $\Mop$ more complicated and then the minimization of \cref{eq:LLR} can become more difficult. Some overlapping patching strategies have been proposed, in particular for dynamic MRI \cite{candes2013unbiased,meyer2023locally}.

\section{Multiresolution low-rank method}
In the proposed \emph{multiresolution low-rank} method the (patched) Casorati matrix is not formed from the image sequence $\f = (f_t)_{t=1}^T$ but from the corresponding 2D \emph{wavelet decompositions} $(\W f_t)_{t=1}^T$. The coefficients from the wavelet decomposition retain some of the local spatial information from $\f$, but the smoothness of $\f$ is no longer tied to the values of neighboring matrix elements but the underlying wavelet system. In other words, even though the patched Casorati matrices split the input into independent, non-overlapping blocks, once the wavelet coefficients are reconstructed (via the inverse wavelet transform $\W^{-1} = \W^*$), the end result should not show sharp artifacts at the edges of the patches since the wavelet elements $\psi_{j,k}$ can overlap with nearby wavelets.

Moreover, since the wavelet transform covers multiple spatial scales, the local low-rank approximation of said transform will also cover multiple scales. Since we are free to change the patch size for the different decomposition levels (and directional subbands, but we omit this), this gives more control than the direct low-rank approximations considered earlier.

Formally we consider the 2D discrete wavelet transform $\W$ and write
\begin{align}
    \W : f &\longmapsto \left( c_j \right)_{j=0}^J = \bm{c},
\end{align}
where the scale $j=0$ corresponds to the scaling coefficients and the scales $j \geqslant 1$ the detail coefficients. The different translates $\ell = (l_1, l_2) \in \Z^2$ are given as
\begin{align}
    c_j[\ell, d] = \langle f, \psi^{(d)}_{j,\ell} \rangle = \langle f, 2^{j} \psi^{(d)}(2^j \cdot - \ell) \rangle,
\end{align}
where the different directional wavelets are constructed from 1D wavelet and scaling functions in the following manner:
\begin{align*}
    \text{horizontal:} \ \psi^{HL}(x_1, x_2) &= \psi(x_1)\phi(x_2), \\
    \text{vertical:} \ \psi^{LH}(x_1, x_2) &= \phi(x_1)\psi(x_2), \\
    \text{diagonal:} \ \psi^{HL}(x_1, x_2) &= \psi(x_1)\psi(x_2).
\end{align*}
The 2D scaling function is denoted $\psi^{LL}(x_1, x_2) = \phi(x_1)\phi(x_2)$.

The multiresolution low-rank (MRLR) method is given by
\begin{align} \label{eq:MRLR_nuclear}
    {M\!RLR} \big( \bm{f} \big) &= \big\| \bm{\A} \bm{f} - \bm{m} \big\|^2_2 + \lambda \sum_{j=0}^J \sum_{k=1}^{K_j} \left\| \big( \Mop_{\p(j,k)} \W \bm{f} \big)_{k} \right\|_*,
\end{align}
where at every scale $j$, the wavelet subbands are broken in to $K_j$ local patches of size $\p_j(k) = p^{j,k}_1 \times p^{j,k}_2$. Again the patch sizes could vary, but for simplicity let us consider square $p_j \times p_j$ patches chosen solely based on the scale $j$. Unlike \cite{liu2023image}, the different directional subbands are kept in separate Casorati matrix patches (accounted by the index $k$).

This regularization term can be written as
\begin{align} \label{eq:MRLR_op}
    \textit{MRLR} \big( \bm{f} \big) &= \big\| \bm{\A} \bm{f} - \bm{m} \big\|^2_2 + \lambda  \left\| \Lop_{\W, \p} \bm{f} \right\|_1.
\end{align}
Here we define
\begin{align*}
    \Lop_{\W, \p} : (f_t)_{t=1}^T &\longmapsto \begin{bmatrix} \Lop_{\p(0)} (\W \bm{f})_0 \\
    \Lop_{\p(1)} (\W \bm{f})_1 \\
    \vdots \\
    \Lop_{\p(J)} (\W \bm{f})_J
    \end{bmatrix},
\end{align*}
where each $\Lop_{\p(j)} (\W \f)_j$ itself is a local low-rank decomposition in to $K_j$ vectors of singular values.

Since the wavelet transform corresponding to orthogonal wavelets is unitary, the operator $\Lop_{\W, \p}$ is also unitary. The process is illustrated in \cref{fig:mr-lr_scheme}.

\begin{figure}[bth]
    \centering
    \begin{tikzpicture}[scale=1]
    
    \begin{scope}[scale=0.5] 
    \foreach \x/\w/\o in {6cm/50/1, 3cm/30/1, 0cm/10/1}{ %
        \filldraw[gray!\w!white, draw=black, opacity=\o] ($(0.15*\x,\x)$) rectangle +(4cm,-4cm);
        
        \draw[very thick, ->, >=stealth, gray!\w!black] ($(0.15*\x + 5cm,\x - 2cm)$) -- +(1.2,0) node[above, pos=.5] {$\W$};

        \begin{scope}[xshift={0.15*\x + 8cm}, yshift={\x + 0.5}]
        \foreach \j/\decSz/\col in {1/20mm/red, 2/10mm/purple, 3/5mm/blue}{ 
        \foreach \k/\l in {0/1, 1/1, 1/0}{ 
            \filldraw[gray!\w!white, draw=white!\w!\col, opacity=\o] ($1.05*(\k*\decSz,-\l*\decSz)$) rectangle +($0.9*(\decSz,-\decSz)$);
            \if\j1
                \draw[dashed, step=10mm, draw=white!\w!\col, opacity=\o] ($1.05*(\k*\decSz,-\l*\decSz)$) grid +($0.9*(\decSz,-\decSz)$);
            \fi
        }
        }
        \filldraw[gray!\w!white, draw=cyan, opacity=\o] (0,0) rectangle +($(5mm,-5mm)$);
        \end{scope}
    } 
    \node at (2.8cm,6.8cm) {$f_t, \ t = 1,\dots,T$};
    \node at (10.8cm,6.8cm) {$\W f_t, \ t = 1,\dots,T$};

    \draw[decorate, decoration = {calligraphic brace}, thick] (12.1cm,-3.0cm) -- +(0,-9mm) node[right, pos=0.5] {\tiny $p_{\!j}$};
    \draw[decorate, decoration = {calligraphic brace}, thick] (11.9cm,-4.0cm) -- +(-9mm,0) node[below, pos=0.5] {\tiny $p_{\!j}$};
    \end{scope}

    \begin{scope}[xshift=10cm, yshift=-1.2cm] 
    \foreach \col/\num/\y/\j in {red/6/1cm/3, purple/3/2cm/2, blue/3/3cm/1, cyan/1/4cm/0}{
        \begin{scope}[yshift={\y}, scale=0.15]
        \foreach \i in {1, ..., \num}{
            \coordinate (p) at ($-\i*(2cm, 1cm)$);
            \ifnum\i=1
                \node[anchor=west] at (1.2cm, -3cm) {\tiny $j=\j$};
            \fi
            \ifnum\ifnum\i<3 1\else\if\i\num 1\else 0\fi\fi=1
                \foreach \x/\w in {0cm/10, 1cm/30, 2cm/50}{ 
                    \fill[gray!\w!white] ($(p) + (\x, 0)$) rectangle +(1cm,-4cm);
                }
                \draw[step=1cm, thin] (p) grid +(3cm,-4cm);
                \draw[thick, \col!80!gray] (p) rectangle +(3cm,-4cm);
            \else\if\i4
                \node[rotate=30] at ($(p) + (1.4cm, -2.5cm)$) {\dots};
                \fi
            \fi
        }
        \end{scope}
    }
    \draw[very thick, ->, >=stealth, black] (-3,1.6) -- +(1.2,0) node[above, pos=.5] {$\Mop_{\p(j)}$} node[below, pos=.5, text width=2.5cm, scale=0.75] {$j = 0,1,\dots,J$; $k = 1,\dots,K_j$};
    \node at (-0.8cm, 4.6cm) {\small $F_{j,k}, \ k = 1, \dots, K_j$};
    \node[anchor=center, font=\small] at (0.65cm, 4.0cm) {scale:};
    \node[anchor=east, text width=2cm, font=\small, text centered, scale=0.75] at (-0.5cm, 3.6cm) {approximation coefficients};
    
    \draw[decorate, decoration = {calligraphic brace}, thick] (-1.9cm,-0.5cm) -- +(0,6mm) node[left, pos=0.5] {\tiny $p_{\!j}^2$};
    \draw[decorate, decoration = {calligraphic brace}, thick] (-1.32cm,-0.6cm) -- +(-5mm,0) node[below, pos=0.5] {\tiny $T$};

    \end{scope}
    
    \end{tikzpicture}
    \caption{Illustration of the multiresolution low-rank scheme using $J=3$ scales. Note how there is a lot of flexibility in choosing the patch sizes for each scale and directional subband.}
    \label{fig:mr-lr_scheme}
\end{figure}
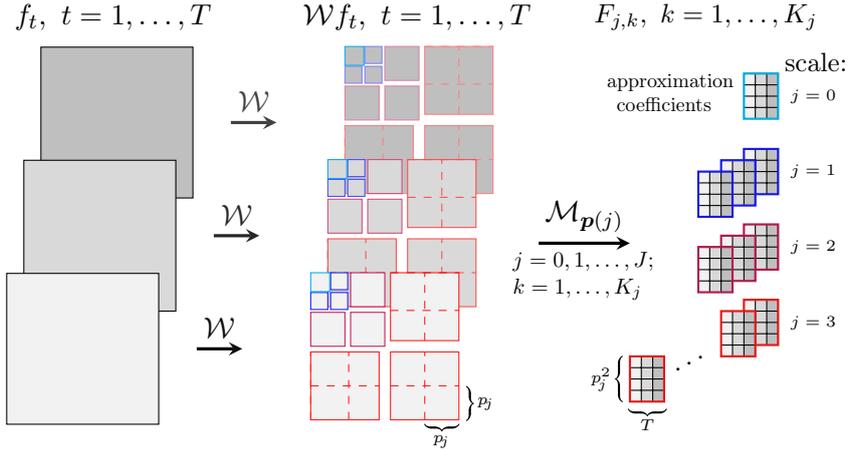

\section{Numerical examples}
To assess the quality and viability of the different low-rank methods the following numerical tests are performed on dynamic X-ray tomography problems.

We consider an ill-posed inverse problem of reconstructing moving object of interest from a single measurement organized into a \emph{sinogram}. Since the underlying movement is unknown, we approximate the problem as a sequence of static \emph{limited angle} problems which are notoriously difficult to reconstruct and often suffer from streaking artifacts.

In addition to the proposed MRLR-method given in \eqref{eq:MRLR_op}, all tested data is also reconstructed using three existing low-rank methods. First the Local Low-Rank (LLR) method \cite{trzasko2011local} as formulated in \eqref{eq:LLR}. Then the Low-rank + Sparse (L+S) method \cite{otazo2015low,gao2011robust}, which is a different kind regularization scheme which decomposes the unknown dynamics into \emph{separate} low-rank and sparse components (using some multiresolution representation system such as wavelets). The objective functional for the L+S method is given as
\begin{equation} \label{eq:L+S}
    \textit{LplusS} \big( L, S \big) = \| \bm{\A}(L + S) - \bm{m} \|_2^2 + \lambda_L \| \Mop L \|_* + \lambda_S \| \W S \|_1,
\end{equation}
where $L,S \in \R^{NT}$, $\Mop$ is the operator defined earlier in \cref{eq:casorati} and $\W$ is again a 2D wavelet transform applied separately to each time step $S_t$ as in \eqref{eq:MRLR_nuclear}. Once again, with a suitable unitary operator, the nuclear norm can be replaced with the $\ell^1$-norm. The two regularization terms have separate regularization parameters $\lambda = (\lambda_L, \lambda_S)$. 
And finally the Fast algorithm for Total Variation and Nuclear Norm Regularization (FTVNNR) \cite{yao2015accelerated,yao2018efficient}. Its objective functional can be given as
\begin{equation}\label{eq:TVNN}
    \textit{FTVNNR} \big( \bm{f} \big) = \| \bm{\A} \bm{f} - \bm{m} \|_2^2 + \lambda_1 \sum_{t=1}^T {TV}(f_t) + \lambda_2 \| \Mop \bm{f} \|_*.
\end{equation}
In short, it incorporates a spatial anisotropic total variation penalty on each time step in addition to the global low-rank penalty term. The implementation is based on the codes available on \url{github.com/uta-smile/FTVNNR_Dynamic_MRI_MEDIA}, and only slightly modified for CT.

In \cite{lima2019sparsity} the authors propose using both locally low-rank and (wavelet) sparsity regularization terms on the unknown image sequence $\bm{f}$ for MRF. This is more closely related to the earlier LLR method in \eqref{eq:LLR} than the L+S method, since the locally low-rank and sparse components are not separated.

The chosen patch sizes and different regularization and wavelet parameters are all listed in table~\ref{tab:parameters}. The LLR and MRLR functionals are minimized using the primal-dual fixed point (PDFP) algorithm \cite{chen2013primal} with non-negativity constraints on the attenuation values. All algorithms are ran until the relative change between iterates drops below $ 5\cdot 10^{-4}$. Daubechies-3 wavelets \cite{daubechies1992ten} are used with every wavelet-based method and data, with periodic boundary conditions to ease the choice of subband patch sizes. The Matlab implementation of the L+S algorithm is based on the example code from \cite{heikkila2022stempoData}. The repository with all the Matlab codes is available on \url{github.com/tommheik/WaveletLowRank}.

\begin{table}[ht!]
    \centering
    \hspace*{-3em}%
    \begin{tabular}{clm{0.15\textwidth} m{0.18\textwidth}cc}
        Data & Method & Wavelet transform & Regularization parameters $\lambda$ & No. of patches & Patch sizes \\[1mm] \hline \\[-2mm]
        \multirow{3}{*}{\rotatebox[origin=c]{90}{Simulated}} & MRLR & db3 ($J = 2$) & $1.0$ & \hspace*{-0.1\textwidth}$[4 \x 4, 2 \x 2, 2 \x 2]$ & $[64 \x 64, 64 \x 64, 64 \x 64]$ \\[1mm]
        & LLR & \ - & $0.1$ & $32 \times 32$ & $8 \times 8$ \\[1mm]
        & L+S & db3 ($J = 3$) & $(0.2, 0.08)$ & - & - \\[1mm]
        & FTVNNR & \ - & $(0.001, 5)$ & - & - \\[1mm] \hline \\[-2mm]
        \multirow{3}{*}{\rotatebox[origin=c]{90}{STEMPO}} & MRLR & db3 ($J = 3$) & $1.0$ & \hspace*{-0.1\textwidth}$[4 \x 4, 2 \x 2, 1 \x 1, 1 \x 1]$ & $[35 \x 35, 35 \x 35, 35 \x 35, 35 \x 35]$ \\[1mm]
        & LLR & \ - & $0.1$ & $40 \times 40$ & $7 \times 7$ \\[1mm]
        & L+S & db3 ($J = 3$) & $(0.35, 0.02)$ & - & - \\[1mm]
        & FTVNNR & \ - & $(0.01, 2)$ & - & - \\[1mm]
    \end{tabular}
    \caption{Table of parameters used}
    \label{tab:parameters}
\end{table}
\vspace{-10pt}

\subsection{Data}

Tests are performed both using simulated and real data:
\smallskip

\textbf{Simulated data} consists of multiple objects of different shapes and attenuation. One object is translated from left to right, the rest undergo periodic and affine deformation, and one object is slightly rotated counterclockwise at the same time. The reconstruction resolution is $256 \times 256$, but the data was initially simulated at twice the spatial resolution and white Gaussian noise of 3\% variance was added to avoid inverse crime. A total of 360 projections with $1^\circ$ steps were computed using 180 unique time steps (i.e. 2 projections per time step).
    
For the reconstructions, the data is organized into just 32 batches, each with 19 consecutive projections and overlap of 11 projections. 
The forward operator follows a parallel beam geometry and is implemented using the ASTRA Toolbox \cite{van2015astra,van2016fast}.

32 reference time steps are chosen (out of the original 180) to match the average state of the batches and some of these are shown in \cref{fig:simu}.

\vspace{-5pt}
\begin{figure}[ht!]
    \setlength{\imSz}{3cm}
    \setlength{\imShift}{1.02\imSz}
    \centering
    \begin{tikzpicture}
        \foreach \mtd/\lbl [count=\j] in {ref_obj/Ground truth}{
        \foreach \t [count=\i] in {1, 12, 22, 32}{
            \node[anchor=center] (img) at ($(\i*\imShift, \j*\imShift)$) {\includegraphics[width=\imSz]{images/\mtd_t\t.png}};
            \if \j 1
                \node[anchor=center, above=-1mm of img] {$t = \t$};
            \fi
        }
        \node[anchor=center, rotate=90] at ($(0.43*\imShift, \j*\imShift)$) {\lbl};
        \node[anchor=center] at ($(4.6*\imShift, \j*\imShift)$) {\includegraphics[height=\imSz]{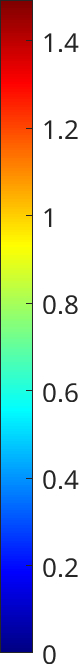}};
        }
    \end{tikzpicture}
    \caption{Selected time steps of the synthetic ground truth object.}
    \label{fig:simu}
\end{figure}

\textbf{Real data} is \texttt{stempo\_cont360\_2d\_b8.mat} from the STEMPO dataset~\cite{heikkila2022stempo,heikkila2022stempoData} which consists of 360 unique fan-beam projections measured at $1^\circ$ steps of a motorized dynamic tomography phantom. Detailed documentation is available in \cite{heikkila2022stempo}. With the chosen binning level (8), the reconstruction resolution is set to $280 \times 280$.

For the reconstructions, the sinogram is organized into just 32 batches, each with 19 consecutive projections and overlap of 11 projections, just like the simulated data. The forward operator is implemented using the ASTRA Toolbox.

\subsection{Results}

Reconstructions of the simulated data are shown in \cref{fig:all_recons_simu} and reconstructions of the STEMPO data are shown in \cref{fig:all_recons_stempo}. In \cref{fig:all_recons_simu} the relative $L^2$-error, SSIM~\cite{wang2004image} and HaarPSI \cite{reisenhofer2018haar} values for each individual time step are also shown.

With both data, all low-rank methods show similar movement artifacts and have particular difficulty reconstructing the translated object (e.g. object at the bottom center in \cref{fig:all_recons_simu}). However, with the chosen wavelet decomposition and block sizes, the MRLR regularization does not produce visible block-artifacts unlike the traditional LLR method. The relative $L^2$-error and HaarPSI metrics also favor MRLR whereas the SSIM values are consistently better for LLR. FTVNNR performs almost as well and the images have less noise thanks to the TV. However the images are also more blurry. This is especially noticeable with the STEMPO reconstructions. None of the methods reconstruct the intermediate time steps well and there is a distinct jump in the translations seen between time steps 12 and 22.

The L+S method worked surprisingly poorly with the tested data both visually and based on the numerical error metrics. Before the non-negativity projection was included in the implementation (without any mathematical rigor), its results were even poorer. Due to the two distinct regularization parameters (see \eqref{eq:L+S}), it is somewhat harder to fine tune than the other methods which might also affect the results. The MRLR also requires choosing the wavelet decomposition and different patch sizes but the changes have clearer effect on the behavior.

Although the MRLR should be slower to compute than the LLR, in practice the difference in computational times seems negligible. 

\begin{figure}[t]
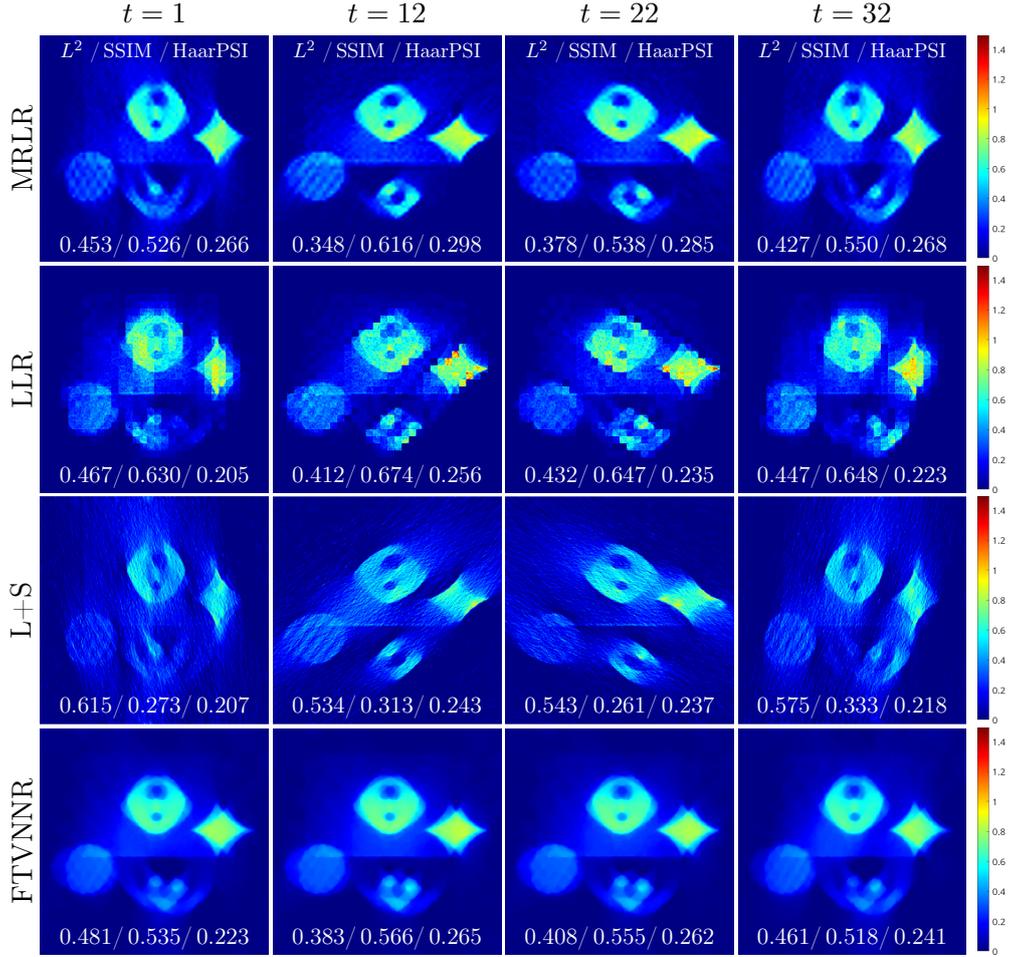


    \def\RElist{{
     {0.481,0.383,0.408,0.461}, 
     {0.615,0.534,0.543,0.575},
     {0.467,0.412,0.432,0.447},
     {0.453,0.348,0.378,0.427},
    }}
    \def\SSIMlist{{
     {0.535,0.566,0.555,0.518}, 
     {0.273,0.313,0.261,0.333},
     {0.630,0.674,0.647,0.648},
     {0.526,0.616,0.538,0.550},
    }}
    \def\HPSIlist{{
     {0.223,0.265,0.262,0.241}, 
     {0.207,0.243,0.237,0.218},
     {0.205,0.256,0.235,0.223},
     {0.266,0.298,0.285,0.268},
    }}
    \setlength{\imSz}{3cm}
    \setlength{\imShift}{1.02\imSz}
    \centering
    \begin{tikzpicture}
        \foreach \mtd/\lbl [count=\j] in {TVNN_recn/FTVNNR, L+S_recn/L+S, LLR_recn/LLR, LMRLR_recn/MRLR}{ 
        \foreach \t [count=\i, evaluate=\e using {\RElist[\j-1][\i-1]}, evaluate=\ss using {\SSIMlist[\j-1][\i-1]}, evaluate=\hpsi using {\HPSIlist[\j-1][\i-1]}] in {1, 12, 22, 32}{
            \node[anchor=center] (img) at ($(\i*\imShift, \j*\imShift)$) {\includegraphics[width=\imSz]{images/\mtd_t\t.png}};
            \node[color=white, scale=0.8]  at ($(\i*\imShift, \j*\imShift - 0.42*\imShift)$) {\e  /\! \ss /\! \hpsi};
            \if \j 4
                \node[color=white, scale=0.7]  at ($(\i*\imShift, \j*\imShift + 0.42*\imShift)$) {$L^2$  /\! SSIM /\! HaarPSI};
                \node[anchor=center, above=-1mm of img] {$t = \t$};
            \fi
        }
        \node[anchor=center, rotate=90] at ($(0.43*\imShift, \j*\imShift)$) {\lbl};
        \node[anchor=center] at ($(4.6*\imShift, \j*\imShift)$) {\includegraphics[height=\imSz]{images/colorbar.png}};
        }
    \end{tikzpicture}
    \caption{Simulated data reconstructions using different low rank methods: MultiResolution Low-Rank, Local Low-Rank, Low-rank + Sparse decomposition and Total Variation and Nuclear Norm regularization. The $L^2$-error, SSIM and HaarPSI values of each time step is also shown.}
    \label{fig:all_recons_simu}
\end{figure}

\begin{figure}[t!]
    \setlength{\imSz}{3cm}
    \setlength{\imShift}{1.02\imSz}
    \centering
    \begin{tikzpicture}
        \foreach \mtd/\lbl [count=\j] in {TVNN_recn/FTVNNR, L+S_recn/L+S, LLR_recn/LLR, LMRLR_recn/MRLR}{
        \foreach \t [count=\i] in {1, 12, 22, 32}{
            \node[anchor=center] (img) at ($(\i*\imShift, \j*\imShift)$) {\includegraphics[width=\imSz]{images/stempo_\mtd_t\t.png}};
            \if \j 4
                \node[anchor=center, above=-1mm of img] {$t = \t$};
            \fi
        }
        \node[anchor=center, rotate=90] at ($(0.43*\imShift, \j*\imShift)$) {\lbl};
        \node[anchor=center] at ($(4.6*\imShift, \j*\imShift)$) {\includegraphics[height=\imSz]{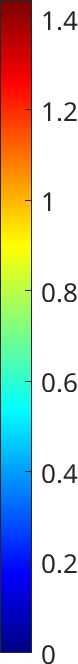}};
        }
    \end{tikzpicture}
    \caption{STEMPO data reconstructions using different low rank methods: MultiResolution Low-Rank, Local Low-Rank, Low-rank + Sparse decomposition and Total Variation and Nuclear Norm regularization.}
    \label{fig:all_recons_stempo}
\end{figure}

\section{Conclusions}
In this paper the multiresolution low-rank decomposition is formulated as a possible regularization strategy for reconstructing dynamic image sequences. As a combination of discrete wavelet transform and a local low-rank decomposition, it forms a unitary operator and should be easily applicable to many optimization strategies. The numerical examples indicate that heavy regularization produces less severe artifacts compared to traditional low-rank regularization methods. However, in practice it still has very similar regularizing effect on the reconstructions and should be applied in place of traditional low-rank methods or in tandem with some different prior. In the future it would be interesting to pair it with stronger spatial regularization such as total variation \cite{rudin1992nonlinear,kazantsev20154d,yao2018efficient}.

\subsubsection*{Acknowledgments} 
The author is grateful of the funding from the Vilho, Yrjö and Kalle Väisälä Foundation;
the Finnish Foundation for Technology Promotion;
Academy of Finland (the Centre of Excellence in Inverse Modelling and Imaging, decision 312339);
and Research Council of Finland (Flagship of Advanced Mathematics for Sensing Imaging and Modelling grant 359183).

\bibliographystyle{splncs04}
\bibliography{refs}
\end{document}